\documentclass[12pt]{article}
\usepackage{amsmath,amssymb,amsthm, comment}

\newtheorem*{theorem*}{Theorem}
\newtheorem{theorem}{Theorem}[section]

\newtheorem{lemma}[theorem]{Lemma}

\def\barr{\begin{array}}
\def\earr{\end{array}}

\title{A generalization of a result on the sum of element orders of a finite group}
\author{Marius T\u arn\u auceanu}
\date{January 21, 2020}

\begin{document}

\maketitle

\begin{abstract}
Let $G$ be a finite group and let $\psi(G)$ denote the sum of element orders of $G$. It is well-known that the
maximum value of $\varphi$ on the set of groups of order $n$, where $n$ is a positive integer, will occur
at the cyclic group $C_n$. For nilpotent groups, we prove a natural generalization of this result,
obtained by replacing the element orders of $G$ with the element orders relative to a certain subgroup $H$ of $G$.
\end{abstract}
\vspace{1mm}
{\small
\noindent
{\bf MSC2000\,:} Primary 20D60; Secondary 20D15, 20F18.

\noindent
{\bf Key words\,:} relative element orders, $p$-groups, nilpotent groups.}

\section{Introduction}
Let $G$ be a finite group. In 2009, H. Amiri, S.M. Jafarian Amiri and I.M. Isaacs introduced in their paper \cite{1} the function
\begin{equation}
\psi(G)=\sum_{x\in G}o(x),\nonumber
\end{equation}where $o(x)$ denotes the order of $x$ in $G$. They proved the following basic theorem:

\begin{theorem*}
If $G$ is a group of order $n$, then $\psi(G)\leq\psi(C_n)$, and we have equality if and only if $G$ is cyclic.
\end{theorem*}Since then many authors have studied the properties of the function $\psi(G)$ and its relations with the structure of $G$. We recall only that $\psi$ is multiplicative and $\psi(C_{p^n})=\frac{p^{2n+1}+1}{p+1}$ when $p$ is a prime (see e.g. Lemmas 2.2(3) and 2.9(1) of \cite{2}).\newpage

Given a subgroup $H$ of $G$, in what follows we will consider the function
\begin{equation}
\psi_H(G)=\sum_{x\in G}o_H(x),\nonumber
\end{equation}where $o_H(x)$ denotes the order of $x$ relative to $H$, i.e. the smallest positive integer $m$ such that $x^m\in H$. Clearly, for $H=1$ we have $\psi_H(G)=\psi(G)$.

By replacing $\psi(G)$ with $\psi_H(G)$, we are able to generalize the above theorem for nilpotent groups.

\begin{theorem}
Let $G$ be a nilpotent group of order $n$ and $H$ be a subgroup of order $m$ of $G$. Then
\begin{equation}
\psi_H(G)\leq\psi_{H_m}(C_n),
\end{equation}where $H_m$ is the unique subgroup of order $m$ of $C_n$.
\end{theorem}

Note that the inequality (1) can easily be proved for normal subgroups $H$. Indeed, in this case we have
\begin{equation}
o_H(x)=o(xH) \mbox{ in } G/H,\, \forall\, x\in G\nonumber
\end{equation}and therefore
\begin{equation}
\psi_H(G)=|H|\psi(G/H)\leq m\,\psi(C_{\frac{n}{m}})=\psi_{H_m}(C_n).\nonumber
\end{equation}This also shows that the equality occurs in (1) whenever $H$ is normal and $G/H$ is cyclic.

Finally, we conjecture that Theorem 1.1 is also true for non-nilpotent groups $G$, i.e. it is true for all finite groups $G$.

Most of our notation is standard and will usually not be repeated here. Elementary notions and results on groups can be found in \cite{4}.

\section{Proof of the main result}

Our first lemma collects two basic properties of the function $\psi_H(G)$.

\begin{lemma}
\begin{itemize}
\item[{\rm a)}] If $(G_i)_{i=\overline{1,k}}$ is a family of finite groups having coprime orders and $H_i\leq G_i$, $i=1,...,k$, then
\begin{equation}
\psi_{H_1\times\cdots\times H_k}(G_1\times\cdots\times G_k)=\prod_{i=1}^k\psi_{H_i}(G_i).\nonumber
\end{equation}In particular, if $G$ is a finite nilpotent group, $(G_i)_{i=\overline{1,k}}$ are the Sylow $p_i$-subgroups of $G$ and $H=H_1\times\cdots\times H_k\leq G$, then\newpage
\begin{equation}
\psi_H(G)=\prod_{i=1}^k\psi_{H_i}(G_i).\nonumber
\end{equation}
\item[{\rm b)}] If $G$ is a finite group and $H\trianglelefteq K\leq G$, then
\begin{equation}
\psi_H(G)\leq[K:H]\,\psi_K(G)-|K|+|H|.
\end{equation}In particular, if $G$ is a finite $p$-group and $H\leq K\leq G$ with $|H|=p^m$ and $|K|=p^{m+1}$, then
\begin{equation}
\psi_H(G)\leq p\,\psi_K(G)-p^m(p-1).
\end{equation}
\end{itemize}
\end{lemma}

\begin{proof}
\begin{itemize}
\item[{\rm a)}] Since $G_i$, $i=1,...,k$, are of coprime orders, for every $x=(x_1,...,x_k)\in G_1\times\cdots\times G_k$ we have
$$o_{H_1\times\cdots\times H_k}(x)=\prod_{i=1}^k o_{H_i}(x_i).$$Then
$$\psi_{H_1\times\cdots\times H_k}(G_1\times\cdots\times G_k)=\!\!\!\!\!\sum_{x=(x_1,...,x_k)\in G_1\times\cdots\times G_k}\!\!\!\!\!o_{H_1\times\cdots\times H_k}(x)$$ $$=\sum_{x_1\in G_1}\cdots\sum_{x_k\in G_k}o_{H_1}(x_1)\cdots o_{H_k}(x_k)$$
$$\hspace{3mm}=\prod_{i=1}^k\left(\,\sum_{x_i\in G_i} o_{H_i}(x_i)\right)=\prod_{i=1}^k\psi_{H_i}(G_i),$$as desired.
\item[{\rm b)}] Let $x\in G$. Then $x^{o_K(x)}\in K$ and so $x^{o_K(x)}H\in K/H$, implying that $\left(x^{o_K(x)}H\right)^{[K:H]}=H$. Thus $x^{[K:H]\,o_K(x)}\in H$, which leads to $$o_H(x)\mid\, [K:H]\,o_K(x)$$and consequently
    $$o_H(x)\leq[K:H]\,o_K(x).$$This shows that $$\psi_H(G)=\sum_{x\in G} o_H(x)=\sum_{x\in G\setminus H}\!\! o_H(x)+\sum_{x\in H} o_H(x)$$ $$\hspace{-8mm}\leq[K:H]\!\!\sum_{x\in G\setminus H}\!\! o_K(x)+|H|$$ $$\hspace{18,5mm}=[K:H]\left(\,\sum_{x\in G} o_K(x)-\sum_{x\in H} o_K(x)\!\right)\!\!+|H|$$ $$=[K:H]\left(\psi_K(G)-|H|\right)+|H|$$ $$\hspace{-2mm}=[K:H]\,\psi_K(G)-|K|+|H|,$$completing the proof.
\end{itemize}
\end{proof}

\noindent{\bf Remark.} By taking $H=1$ and $K\trianglelefteq G$ in (2), one obtains
\begin{equation}
\psi(G)\leq|K|\,\psi_K(G)-|K|+1=|K|^2\,\psi(G/K)-|K|+1.
\end{equation}This improves the inequality in Proposition 2.6 of \cite{3}. Also, by taking $K=G$ in (4), we get a new upper bound for $\psi(G)$:
\begin{equation}
\psi(G)\leq|G|^2-|G|+1.
\end{equation}Note that we have equality in (5) if and only if $G$ is cyclic of prime order.
\bigskip

Next we prove the inequality (1) for $p$-groups.

\begin{lemma}
Let $G$ be a $p$-group of order $p^n$ and $H$ be a subgroup of order $p^m$ of $G$. Then
\begin{equation}
\psi_H(G)\leq\psi_{H_{p^m}}(C_{p^n}).\nonumber
\end{equation}
\end{lemma}

\begin{proof}
We will proceed by induction on $[G:H]$. Obviously, the inequality holds for $[G:H]=1$. Assume now that it holds for all subgroups of $G$ of index $<[G:H]$. Since every subgroup of $G$ is subnormal, we can choose $K\leq G$ such that $H\subset K$ and $|K|=p^{m+1}$. Then, by (3) and the inductive hypothesis, we get
$$\hspace{-15mm}\psi_H(G)\leq p\,\psi_K(G)-p^m(p-1)\leq p\,\psi_{H_{p^{m+1}}}(C_{p^n})-p^m(p-1)$$
$$\hspace{20,3mm}=p^{m+2}\psi(C_{p^{n-m-1}})-p^m(p-1)=p^m\left[p^2\,\frac{p^{2n-2m-1}+1}{p+1}-(p-1)\right]$$
$$\hspace{-9,5mm}=p^m\,\frac{p^{2n-2m+1}+1}{p+1}=p^m\psi(C_{p^{n-m}})=\psi_{H_{p^m}}(C_{p^n}),$$as desired.
\end{proof}

We are now able to prove our main result.

\bigskip\noindent{\bf Proof of Theorem 1.1.} Let $n=p_1^{n_1}\cdots p_k^{n_k}$ be the decomposition of $n$ as a product of prime factors. Since $G$ is nilpotent, we have $G\cong G_1\times\cdots\times G_k$, where $(G_i)_{i=\overline{1,k}}$ are the Sylow $p_i$-subgroups of $G$. Moreover, any subgroup $H$ of $G$ is of type $H\cong H_1\times\cdots\times H_k$ with $H_i\leq G_i$, $|H_i|=p_i^{m_i}$, $\forall\, i=1,...,k$. Then, Lemmas 2.1 (a)) and 2.2 lead to
\begin{equation}
\psi_H(G)=\prod_{i=1}^k\psi_{H_i}(G_i)\leq\prod_{i=1}^k\psi_{H_{p_i^{m_i}}}(C_{p_i^{n_i}})=\psi_{H_m}(C_n).\nonumber
\end{equation}This completes the proof.$\qed$

\vspace*{3ex}\small

\hfill
\begin{minipage}[t]{5cm}
Marius T\u arn\u auceanu \\
Faculty of  Mathematics \\
``Al.I. Cuza'' University \\
Ia\c si, Romania \\
e-mail: {\tt tarnauc@uaic.ro}
\end{minipage}

\end{document}